\documentclass[a4paper,reqno]{amsart}

\usepackage{amssymb,amsmath}

\usepackage{paralist,enumitem}
\setlist{listparindent=0pt,parsep=3pt}
\setdefaultenum{1.}{}{}{}

\usepackage[pdftex,linktocpage,pdfstartview=FitH,colorlinks=true,allcolors=blue]{hyperref}
\usepackage{amsthm} 
\usepackage[capitalise]{cleveref}

\usepackage[initials,msc-links]{amsrefs}[2007/10/22]

\newtheorem{thm}{Theorem}[section]
\newtheorem{prop}[thm]{Proposition}
\newtheorem{cor}[thm]{Corollary}
\newtheorem{lem}[thm]{Lemma}
\newtheorem*{proposition*}{Proposition}
\newtheorem*{theorem*}{Theorem}
\theoremstyle{definition}

\newtheorem*{defn}{Definition}

\newtheorem*{xmpl}{Example}

\theoremstyle{remark}
\newtheorem*{rem}{Remark}
\newtheorem*{rems}{Remarks}

\numberwithin{equation}{section}

\newcommand{\R}{\mathbb{R}}

\DeclareMathOperator{\End}{End}
\DeclareMathOperator{\Hom}{Hom}

\newcommand{\I}{\mathrm{I}}
\newcommand{\Span}[1]{\operatorname{span}\{#1\}}
\newcommand{\trace}{\operatorname{trace}}
\DeclareSymbolFont{script}{U}{eus}{m}{n}
\DeclareSymbolFontAlphabet{\mathscr}{script}
\DeclareMathSymbol{\EuWedge}{0}{script}{"5E}
\newcommand{\Wedge}{\EuWedge}
\renewcommand{\d}{\mathop{}\!\mathrm{d}}
\newcommand{\tens}{\otimes}
\newcommand{\Dif}{\mathrm{D}}

\DeclareMathOperator{\rSL}{SL}

\newcommand{\fo}{\mathfrak{o}}
\newcommand{\fso}{\mathfrak{so}}

\newcommand{\fe}{\mathfrak{e}}
\newcommand{\conf}{\mathfrak{c}}

\newcommand{\g}{\mathfrak{g}}

\renewcommand{\q}{\mathfrak{q}}

\newcommand{\Ad}{\operatorname{Ad}}

\newcommand{\II}{\mathrm{I\kern-0.5ptI}}
\renewcommand{\I}{\mathrm{I}}
\newcommand{\mcv}{\mathbf{H}}
\newcommand{\sect}[1]{\Gamma\,#1}

\newcommand{\st}{\mathrel{|}}
\newcommand{\set}[1]{\{#1\}}
\newcommand{\vol}{\mathrm{vol}}
\newcommand{\restr}[1]{{}_{|#1}}
\renewcommand{\Re}{\operatorname{Re}}
\newcommand{\ev}[1]{\langle#1\rangle}
\newcommand{\stab}{\mathrm{stab}}
\renewcommand{\lor}[1][4,1]{\R^{#1}}
\newcommand{\cL}{\mathcal{L}}
\newcommand{\ip}[1]{(#1)}
\newcommand{\B}[1]{(#1)_{\g}}
\newcommand{\Y}{Y}
\newcommand{\Sym}{\operatorname{Sym}}
\newcommand{\dual}[1]{\check{#1}}
\newcommand{\xh}{\dual{x}}
\DeclareMathOperator{\Diff}{Diff}

\newcommand{\abs}[1]{\lvert#1\rvert}

\newcommand{\half}{\tfrac12}

\newcommand{\TitleWithUrl}[1]{\IfEmptyBibField{doi}%
  {\IfEmptyBibField{url}{\textit{#1}}%
    {\IfEmptyBibField{eprint}{\href {\BibField{url}}{\textit{#1}}}{\textit{#1}}}%
    }%
  {\href {https://doi.org/\BibField{doi}}{\textit{#1}}}}
\renewcommand{\eprint}[1]{\IfEmptyBibField{url}{\url{#1}}%
  {\href {\BibField{url}}{#1}}}

\BibSpec{article}{%
    +{}  {\PrintAuthors}                {author}
    +{,} { \TitleWithUrl}               {title}
    +{.} { }                            {part}
    +{:} { \textit}                     {subtitle}
    +{,} { \PrintContributions}         {contribution}
    +{.} { \PrintPartials}              {partial}
    +{,} { }                            {journal}
    +{}  { \textbf}                     {volume}
    +{}  { \PrintDatePV}                {date}
    +{,} { \issuetext}                  {number}
    +{,} { \eprintpages}                {pages}
    +{,} { }                            {status}
    +{,} { available at \eprint}        {eprint}
    +{}  { \parenthesize}               {language}
    +{}  { \PrintTranslation}           {translation}
    +{;} { \PrintReprint}               {reprint}
    +{.} { }                            {note}
    +{.} {}                             {transition}
    +{}  {\SentenceSpace \PrintReviews} {review}
}

\BibSpec{collection.article}{%
    +{}  {\PrintAuthors}                {author}
    +{,} { \TitleWithUrl}                     {title}
    +{.} { }                            {part}
    +{:} { \textit}                     {subtitle}
    +{,} { \PrintContributions}         {contribution}
    +{,} { \PrintConference}            {conference}
    +{}  {\PrintBook}                   {book}
    +{,} { }                            {booktitle}
    +{,} { \PrintDateB}                 {date}
    +{,} { pp.~}                        {pages}
    +{,} { }                            {status}
    +{,} { available at \eprint}        {eprint}
    +{}  { \parenthesize}               {language}
    +{}  { \PrintTranslation}           {translation}
    +{;} { \PrintReprint}               {reprint}
    +{.} { }                            {note}
    +{.} {}                             {transition}
    +{}  {\SentenceSpace \PrintReviews} {review}
}
\BibSpec{book}{%
    +{}  {\PrintPrimary}                {transition}
    +{,} { \TitleWithUrl}               {title}
    +{.} { }                            {part}
    +{:} { \textit}                     {subtitle}
    +{,} { \PrintEdition}               {edition}
    +{}  { \PrintEditorsB}              {editor}
    +{,} { \PrintTranslatorsC}          {translator}
    +{,} { \PrintContributions}         {contribution}
    +{,} { }                            {series}
    +{,} { \voltext}                    {volume}
    +{,} { }                            {publisher}
    +{,} { }                            {organization}
    +{,} { }                            {address}
    +{,} { \PrintDateB}                 {date}
    +{,} { }                            {status}
    +{}  { \parenthesize}               {language}
    +{}  { \PrintTranslation}           {translation}
    +{;} { \PrintReprint}               {reprint}
    +{.} { }                            {note}
    +{.} {}                             {transition}
    +{}  {\SentenceSpace \PrintReviews} {review}
}

\BibSpec{thesis}{%
    +{}  {\PrintAuthors}                {author}
    +{.} { \PrintDate}                  {date}
    +{.} { \TitleWithUrl}               {title}
    +{:} { \textit}                     {subtitle}
    +{,} { \PrintThesisType}            {type}
    +{,} { }                            {organization}
    +{,} { }                            {address}
    +{,} { \eprint}                     {eprint}
    +{,} { }                            {status}
    +{}  { \parenthesize}               {language}
    +{}  { \PrintTranslation}           {translation}
    +{;} { \PrintReprint}               {reprint}
    +{.} { }                            {note}
    +{.} {}                             {transition}
    +{}  {\SentenceSpace \PrintReviews} {review}
}

\title{Isothermic surfaces and conservation laws}

\author{F.E. Burstall}
\address{Department of Mathematical Sciences\\ University of Bath\\
  Bath BA2 7AY\\UK}
\email{feb@maths.bath.ac.uk}

\author{E. Carberry}
\address{%
School of Mathematics and Statistics\\The University of Sydney\\NSW 2006\\Australia}
\email{emma.carberry@sydney.edu.au}

\author{U. Hertrich-Jeromin}
\address{%
  Institute of Discrete Mathematics and Geometry\\ 
  TU Wien\\
  Wiedner Hauptstrasse 8-10/104\\
  1040 Wien\\
  Austria}
\email{udo.hertrich-jeromin@tuwien.ac.at}

\author{F. Pedit}
\address{Department of Mathematics and Statistics\\ University of Massachusetts Amherst\\
  Amherst, MA 01030\\USA}
\email{pedit@math.umass.edu}

\dedicatory{In memory of Joe Wolf}

\thanks{We thank MATRIX for their hospitality and for
  creating an environment so conducive to fruitful
  collaboration. The first and fourth authors further
  gratefully acknowledge support from MATRIX and the Simons
  Foundation in the form of a Travel Grant and Family
  Funding, respectively. The third and fourth authors were
  partially supported by the International Visitor Programme
  of the Sydney Mathematical Research Institute (SMRI) and
  thank the members of the Institute for their welcoming and
  inspiring research environment.}
\begin{document}

\begin{abstract}
  For CMC surfaces in $3$-dimensional space forms, we relate
  the moment class of Korevaar--Kusner--Solomon to a second
  cohomology class arising from the integrable systems
  theory of isothermic surfaces.  In addition, we show that
  both classes have a variational origin as Noether
  currents.
\end{abstract}
\maketitle

\section{Introduction}
\label{sec:introduction}

The purpose of this paper is to record and explore a puzzling
coincidence that we observed during a MATRIX workshop.  The
setting is classical differential geometry of surfaces and
the coincidence relates two Lie algebra valued cohomology
classes associated to a constant mean curvature (CMC)
surface in a $3$-dimensional space form.

The first of these classes is the \emph{moment class} of
Kusner and his collaborators \cite{MR1081331,KorKus93,KorKusSol89} of
which the following is a paradigm: let $\Sigma\subset\R^3$
have constant mean curvature $H$, $Y$ be a Killing
field on $\R^3$ and $\gamma$ a $1$-cycle in $\Sigma$
bounding a $2$-cycle $D$ in $\R^3$.  Then
\begin{equation*}
  \int_{\gamma}\ip{Y,\nu}-2H\int_D\ip{Y,N}
\end{equation*}
\emph{depends only on the homology class of $\gamma$}.  Here
$\nu,N$ are appropriate unit normals to $\gamma$ in $\Sigma$
and $D$ in $\R^3$, respectively.  Letting $Y$ vary in the
Lie algebra $\g_{\R^3}$ of Killing fields, we arrive at a
cohomology class $\mu\in H^1(\Sigma,\R)\tens\g_{\R^3}^{*}$,
the \emph{moment class}.  This has been used to demonstrate
the necessity of the balancing conditions that Kapouleas
\cite{Kap90} needed in his celebrated existence theorem for
CMC surfaces.  A similar analysis can be carried out for CMC
immersed hypersurfaces in any Riemannian manifold $M^{k+1}$ under mild
topological conditions to produce $\mu\in
H^{k-1}(\Sigma,\R)\tens\g_{M}^{*}$, with $\g_M$ the Killing
algebra of $M$.

Returning to the $\R^3$ case, Meeks--P\'erez--Tinaglia
\cite{MR3525098} observe that the restriction
$\mu\restr{\R^3}$ to the constant vector fields has a
de~Rham representative given by the \emph{flux form}:
\begin{equation*}
  \d x\times (N+Hx)\in\Omega_{\Sigma}^1\tens\R^{3}
\end{equation*}
where $x\colon\Sigma\to\R^3$ is the inclusion and $N$ the normal
to $x$.  The starting point of the present paper is the
observation that the flux form essentially appears as a component
of another closed Lie algebra valued form, the
\emph{retraction form}, which is the corner-stone of the
modern theory of isothermic surfaces.

Recall that an immersed surface $x\colon\Sigma\to\R^3$ is
\emph{isothermic} if, away from umbilics, it locally admits
conformal coordinates $u,v$ which are simultaneously
curvature line coordinates.  This amounts to the local
existence of holomorphic quadratic differentials that
commute with the second fundamental form $\II$ of $x$.  We
say that $x$ is globally isothermic \cite{Smy04} if there is
a globally defined holomorphic quadratic differential
$q\in H^0(\Sigma,K^2)$ which so commutes.  CMC surfaces are
globally isothermic since then $\II^{2,0}$ is holomorphic by
the Codazzi equation.

Globally isothermic surfaces comprise an integrable system
\cite{CieGolSym95}.  This can, in large part, be understood
as a consequence of the existence of a pencil of flat
connections with gauge potential $\eta$ which is a closed
$1$-form with values in the Lie algebra $\conf_{\R^3}$ of
\emph{conformal} vector fields on $\R^3$.  This
$\eta\in\Omega^1_{\Sigma}\tens\conf_{\R^3}$, which is
constructed from $x$ and $q$, is the retraction form.  When
$x$ is a CMC surface, the component of $\eta$ in
$\fso(3)\cong\Wedge^2\R^3$, the infinitesimal rotations
about $0$, is cohomologous to
\begin{equation*}
  \d x\wedge(N+Hx)
\end{equation*}
which coincides with the flux form after using the cross
product to identify $\Wedge^2\R^3$ with $\R^3$.  The alert
reader will spot that something unusual is going on since
we are identifying the rotational part of $\eta$ with the
(dual of) the translational part of $\mu$.

In our attempts to understand this, we arrived at the
following context and results.  Let $M$ be a $3$-dimensional
space form with $\g_M$ the Lie algebra of Killing fields on
$M$ and $\conf_M$ the Lie algebra of conformal vector
fields so that $\g_M\leq\conf_{M}$.  Our first observation is:

\emph{There is an isomorphism $S\colon\g_M\cong\g_M^{*}$ which is
  equivariant for the action of the isometry group of $M$.}

Now let $x\colon\Sigma\to M$ be a CMC, and therefore isothermic,
surface with moment class
$\mu\in H^1(\Sigma,\R)\tens\g_M^{*}$ and retraction class
$[\eta]\in H^1(\Sigma,\R)\tens\conf_M$.  Then
\begin{compactenum}
\item \emph{$[\eta]$ takes values in $\g_M$.}
\item $S[\eta]=\mu$.
\end{compactenum}
As corollaries of this analysis, we find that parallel pairs of CMC
surfaces in space forms have the same moment class and
we provide an explicit de~Rham representative of the
moment class which generalises the flux form.

Here is the plan of the paper.  We describe the moment class
and retraction form in Sections \ref{sec:cmc-cons-laws} and
\ref{sec:isothermic-surfaces} respectively.  In Section
\ref{sec:cons-laws-again} we prove our main results.
However, while we provide complete proofs, we lack a
conceptual understanding of \emph{why} our results are true.
In Section~\ref{sec:noethers-theorem}, we attempt a partial
resolution by showing that both moment and retraction
classes have a variational origin as Noether currents.  That
said, a new idea is required to relate the variational
problems involved.

Finally: it was an honour to participate in Joe Wolf's final
workshop. The generosity and gentleness with which he shared
his wisdom was a delight to witness.

\section{A conservation law for CMC hypersurfaces}
\label{sec:cmc-cons-laws}

We begin by setting out an approach to the result of
Korevaar--Kusner--Solomon \cite{KorKusSol89}.

Let $x\colon\Sigma^k\to M^n$ be an immersion of an oriented
$k$-manifold into an oriented Riemannian $n$-manifold.  We
denote the metric on $M$ by $\ip{\,,\,}$.  Equip $\Sigma$
with the induced metric $\I=\ip{\d x,\d x}$ and volume form
$\vol_{x}\in\Omega^k_{\Sigma}$.  With $N\Sigma$ the normal
bundle of $x$, we have an orthogonal decomposition
\begin{equation*}
  x^{*}TM=\d x(T\Sigma)\oplus N\Sigma
\end{equation*}
with a corresponding decomposition of vector fields along
$x$:
\begin{equation*}
  X\circ x=\d x(X^{\top})+X^{\perp},
\end{equation*}
for $X\in\sect{TM}$.  Let
$\II\in\sect{S^2T^{*}\Sigma\tens N\Sigma}$ be the second
fundamental form of $x$ and
$\mcv=\frac1k\trace\II\in\sect{N\Sigma}$ the mean curvature
vector.

With this in hand, we have
\begin{lem}
  \label{th:1}
  For $Y$ a Killing field on $M$,
  \begin{equation}
    \label{eq:1}
  \d i_{Y^{\top}}\vol_x=k(Y\circ x,\mcv)\vol_x.  
  \end{equation}
\end{lem}
\begin{proof}
  Using the Cartan identity we have
  \begin{equation*}
    \d i_{Y^{\top}}\vol_x=L_{Y^{\top}}\vol_x=(\trace \nabla^{\Sigma}Y^{\top})\vol_x,
  \end{equation*}
  where $\nabla^{\Sigma}$ is the Levi-Civita connection on
  $\Sigma$.  On the other hand, with $\nabla^M$ the
  (pullback of) the Levi-Civita connection on $M$ and $A$
  the shape operator, we have
  \begin{equation}\label{eq:8}
    \nabla^{\Sigma}Y^{\top}=(\nabla^MY)^{\top}+A^{Y^{\perp}}.
  \end{equation}
  Now $Y$ is Killing so $(\nabla^MY)^{\top}$ is skew and so
  trace-free while
  \begin{equation*}
    \trace A^{Y^{\perp}}=(Y\circ x,\trace\II)=k(Y\circ x,\mcv).
  \end{equation*}
\end{proof}

When $x$ is a hypersurface so that $n=k+1$, we take
$N\in\sect{N\Sigma}$ to be the unit normal for which
$x^{*}(i_{N}\vol_{M})=\vol_x$ and so define the mean
curvature $H$ by $HN=\mcv$.  Now \eqref{eq:1} reads
\begin{equation}
  \label{eq:2}
  \d i_{Y^{\top}}\vol_x=kH(Y\circ x,N)\vol_x=kHx^{*}(i_Y\vol_M).
\end{equation}

Now suppose that the de~Rham cohomology
$H^k(M,\R)=H^{k-1}(M,\R)=0$.  Then, another application of
the Cartan identity gives
\begin{equation*}
  \d i_Y\vol_M=L_Y\vol_M=0,
\end{equation*}
since $Y$ is Killing.  Our hypotheses on the cohomology of
$M$ now guarantee the existence of
$\alpha_Y\in\Omega_M^{k-1}$, unique up to the addition of an
exact $(k-1)$-form, with
\begin{equation*}
  \d\alpha_Y=i_Y\vol_M.
\end{equation*}
In this case, \eqref{eq:2} becomes
\begin{equation}
  \label{eq:3}
  \d i_{Y^{\top}}\vol_x=kH\d(x^{*}\alpha_{Y}).
\end{equation}
We therefore conclude:
\begin{prop}\label{th:2}
  Let $M^{k+1}$ be an oriented Riemannian manifold with
  $H^k(M,\R)=H^{k-1}(M,\R)=0$, $Y$ a Killing field on $M$
  and $x\colon\Sigma^k\to M^{k+1}$ be an immersed oriented
  hypersurface with \emph{constant} mean curvature $H$.
  Then
  \begin{equation}
    \label{eq:4}
    \d(i_{Y^{\top}}\vol_x-kHx^{*}\alpha_{Y})=0.
  \end{equation}
\end{prop}

Denote by $\mu_x(Y)\in H^{k-1}(\Sigma,\R)$ the cohomology
class of $i_{Y^{\top}}\vol_x-kHx^{*}\alpha_{Y}$ and let $\g_M$
denote the Lie algebra of Killing fields of $M$.  Observe
that the uniqueness of $\alpha_Y$ modulo exact forms ensures
that $Y\mapsto\alpha_{Y}$ induces a linear map
$\g_{M}\to\Omega^{k-1}_M/\d\Omega_M^{k-2}$ so that
$Y\mapsto \mu_x(Y)$ is also a linear map defining the
\emph{moment class}
\begin{equation}\label{eq:21}
  \mu_x\in H^{k-1}(\Sigma,\R)\tens \g_{M}^{*}.
\end{equation}
Let $G_M$ be the isometry group of $M$ and $g\in G_M$.  For
$p\in M$, we have
\begin{equation}\label{eq:5}
  Y(gp)=\d_p g(\Ad(g)^{-1}Y)(p)
\end{equation}
so that
\begin{equation*}
  i_{\Ad(g)^{-1}Y}\vol_M=g^{*}(i_Y\vol_M).
\end{equation*}
It follows at once that
\begin{equation*}
  \alpha_{\Ad(g)^{-1}Y}\equiv g^{*}\alpha_Y\mod\d\Omega^{k-2}_{M}.
\end{equation*}
Again from \eqref{eq:5}, we deduce that
\begin{equation*}
  i_{(\Ad(g)^{-1}Y)^{\top}}\vol_x=i_{Y^{\top}}\vol_{gx}
\end{equation*}
and so finally that
\begin{equation*}
  \mu_{gx}(Y)=\mu_x(\Ad(g)^{-1}Y).
\end{equation*}
Otherwise said:
\begin{prop}[c.f.\ \cite{MR1081331}]\label{th:6}
  The moment class is equivariant for the co-adjoint action
  of $G_M$ on $\g_M^{*}$:
  \begin{equation*}
    \mu_{gx}=\Ad^{*}(g)\mu_{x},
  \end{equation*}
  for all $g\in G_M$.
\end{prop}

As an application of these ideas, we recover the following
conservation law of Korevaar--Kusner--Solomon:
\begin{thm}[\citelist{\cite{MR1081331} \cite{KorKusSol89}*{Remark
    3.17} \cite{MR3525098}*{Theorem~3.4}}]\label{th:3}
  In the situation of \cref{th:2}, further suppose that
  $\Sigma$ is the boundary of a $(k+1)$-manifold $\Omega$
  and that $x$ extends to a local diffeomorphism
  $\hat{x}\colon\Omega\to M$.

  Let $\gamma$ be a hypersurface in
  $\Sigma$ and $D$ a hypersurface in $\Omega$ with
  $\partial D=\gamma$.  Then
  \begin{equation*}
    \int_{\gamma}(Y,\nu)\vol_{x\restr\gamma}-
    kH\int_{D}(Y,N)\vol_{\hat{x}\restr D}
  \end{equation*}
  depends only on the homology class of $\gamma$.  Here
  $\nu$ is the normal to $\gamma$ in $\Sigma$ and $N$ the
  normal to $\hat{x}\restr{D}$.
\end{thm}
\begin{proof}
  We recognise the integrands:
  $(Y,\nu)\vol_{x\restr\gamma}=i_{Y^{\top}}\vol_{x}$ on $\gamma$
  while $(Y,N)\vol_{\hat{x}\restr
    D}=\hat{x}^{*}(i_{Y}\vol_{M})=\d(\hat{x}^{*}\alpha_{Y})$
  on $D$.  Thus, by Stokes' Theorem, the integral is
  \begin{equation*}
    \int_{\gamma}i_{Y^{\top}}\vol_x-kH\int_D\d(\hat{x}^{*}\alpha_{Y}) =
    \int_{\gamma}(i_{Y^{\top}}\vol_x-kH x^{*}\alpha_Y)
  \end{equation*}
  which is simply the evaluation of $\mu_x$ on the homology
  class of $\gamma$.
\end{proof}

We would like to find an explicit closed $(k-1)$-form
representing $\mu_{x}$, or, equivalently, an explicit
representative of the coset of $\alpha_Y$ in
$\Omega_M^{k-1}/\d\Omega_M^{k-2}$.  When $M=\R^{k+1}$ this
is possible.  Indeed, it is an amusing exercise to
check
that 
\begin{equation*}
  \alpha_{Y}\equiv \begin{cases}
    \tfrac{1}{k}i_Xi_Y\vol_{\R^{k+1}} \mod \Omega^{k-2}_{\R^{k+1}},&\text{for $Y$ an infinitesimal translation;}\\
    \tfrac{1}{k+1}i_Xi_Y\vol_{\R^{k+1}} \mod \Omega^{k-2}_{\R^{k+1}},&\text{for $Y$ an
      infinitesimal rotation about $0$,}
  \end{cases}
\end{equation*}
where $X$ is the radial vector field, $X(p)=p$, for $p\in\R^{k+1}$.
However, even when $M=S^{k+1}$, this seems to be difficult.

That said, when $k=2$ and $M$ is a space-form\footnote{That
  is, $M$ is simply connected with constant sectional
  curvatures.}, we can find such a representative by
recourse to the theory of isothermic surfaces, to which we
now turn.

\section{Isothermic surfaces}
\label{sec:isothermic-surfaces}

From now on, we restrict attention to immersions
$x\colon\Sigma^2\to M^n$ of an oriented surface into an oriented
Riemannian manifold.

\begin{defn}[Isothermic surface]
  An immersion $x\colon\Sigma^2\to M^n$ is \emph{isothermic} if,
  away from umbilics, it locally admits conformal curvature
  line coordinates, thus conformal coordinates $u,v$ for
  $\I$ for which the coordinate vector fields
  $\partial/\partial u,\partial/\partial v$ are eigenvectors
  for all shape operators $A$ of $x$.
\end{defn}
\begin{rems}\item[]
  \begin{compactenum}
  \item It follows that all shape operators of an isothermic
    $x$ commute.
  \item Isothermicity is a condition on the trace-free part
    of the shape operators and so is preserved by conformal
    diffeomorphisms: if $x\colon\Sigma\to M$ is isothermic and
    $T\colon M\to \hat{M}$ is a conformal diffeomorphism then
    $T\circ x$ is also isothermic.
  \end{compactenum}
\end{rems}

The conformal structure of the induced metric $\I$ makes
$\Sigma$ into a Riemann surface, so that complex variable
methods can be applied.  This offers a different take on the
isothermic condition: for $u,v$ conformal curvature line
coordinates, $z:=u+iv$ is a holomorphic coordinate and
$q:=\d z^2$ is a holomorphic quadratic differential which
commutes with shape operators in the sense that
\begin{equation}\label{eq:6}
  Q(A^{\nu}U,V)=Q(U,A^{\nu}V),
\end{equation}
for all $U,V\in T\Sigma$ and $\nu\in N\Sigma$, where
$Q=2\Re q$.  This prompts (c.f.\ Smyth \cite{Smy04}):
\begin{defn}[Globally isothermic surface]
  An immersion $x\colon\Sigma^2\to M^n$ is \emph{globally
    isothermic} if there is a holomorphic quadratic
  differential $q\in H^0(K^2)$ for which \eqref{eq:6} holds.
  
  In this case, we say that $(x,q)$ is globally isothermic.
\end{defn}

We remark that globally isothermic surfaces are isothermic:
away from zeroes of $q$, we may locally find a holomorphic
coordinate $z$ with $q=\d z^2$.

\begin{xmpl}
  Let $x\colon\Sigma\to M^3$ take values in a $3$-dimensional
  space form.  In this setting, it is well
  known \cite[p.~137]{MR1013786} that $H$ is constant if and only if
  the \emph{Hopf differential} $\II^{2,0}$ is a holomorphic
  quadratic differential.  Thus \emph{constant mean curvature
  surfaces in $3$-dimensional space forms are globally
  isothermic with $q=\II^{2,0}$.}
\end{xmpl}

\begin{rem}
  For $M=S^n$, there are many isothermic surfaces.  Indeed,
  in this case, they comprise an integrable system
  \cite{CieGolSym95}.  However, we must confess that we know
  of no examples away from this conformally flat setting.
  It would be interesting to know if any CMC surfaces in
  $3$-manifolds with $4$-dimensional isometry group are
  (globally) isothermic (the survey of Fern\'andez--Mira
  \cite{MR2827821} is a good starting place).
\end{rem}

\begin{defn}[Retraction form]
  Let $x\colon\Sigma\to M$ be an immersion, $q\in H^0(K^2)$ and
  $Y$ a conformal vector field on $M$.  Define
  $\eta_Y\in\Omega^1_{\Sigma}$ by
  \begin{equation*}
    \eta_Y(U)=2\Re q(Y^{\top},U)
  \end{equation*}
  and so the \emph{retraction form}
  $\eta_{x}^{q}\in\Omega^1_{\Sigma}\tens\conf_{M}^{*}$, for
  $\conf_M$ the Lie algebra of conformal vector fields on
  $M$, by
  \begin{equation*}
    \eta_{x}^q(Y)=\eta_Y.
  \end{equation*}
\end{defn}
\goodbreak
\begin{rems}
\item{}
  \begin{compactenum}
  \item \emph{$\eta^q$ is equivariant for the co-adjoint action of
    the conformal diffeomorphism group of $M$}: if $g\colon M\to M$
    is a conformal diffeomorphism of $M$ then
    \begin{equation}
      \label{eq:7}
      \eta_{gx}^q=\Ad^{*}(g)\eta^q_x,
    \end{equation}
    thanks to \eqref{eq:5} (which holds for any Lie group
    acting on $M$).
  \item For $p\in M$, set $\stab_{\conf_M}(p)=\set{Y\in\conf_M\st
      Y(p)=0}$, the infinitesimal stabiliser of $p$. Thus,
    if $s\in\Sigma$ and $Y\in\stab_{\conf_M}(x(s))$ is a
    conformal vector field vanishing at $x(s)$,
    $\eta_{x}^q(Y)=2\Re q(Y^{\top},-)$ also
    vanishes at $s$ and we conclude that \emph{$\eta_x^q$
      takes values in the bundle $\stab_{\conf_{M}}(x)^{\circ}$ of annihilators of
      infinitesimal stabilisers along $x$}.
  \end{compactenum}
\end{rems}

Here is the point of this construction:
\begin{prop}
  \label{th:4}
  If $(x,q)\colon\Sigma\to M$ is globally isothermic then $\d\eta_x^{q}=0$.
\end{prop}
\begin{proof}
  We must show that $\d\eta_Y=0$, for $Y\in\conf_M$.  For
  this, write $Q=2\Re q$ and view $Q$ as a
  $T^{*}\Sigma$-valued $1$-form.  Then
  \begin{equation*}
    \eta_Y=\ev{Q,Y^{\top}}
  \end{equation*}
  so that
  \begin{equation*}
    \d\eta_Y=\ev{\d^{\nabla^{\Sigma}}Q,Y^{\top}}-\ev{Q\wedge\nabla^{\Sigma}Y^{\top}},
  \end{equation*}
  where $\ev{\,,\,}$ is the evaluation pairing on
  $T^{*}\Sigma\times T\Sigma$ which we use to multiply
  coefficients in the wedge product.  Using \eqref{eq:8},
  this reads
  \begin{equation*}
    \d\eta_Y=\ev{\d^{\nabla^{\Sigma}}Q,Y^{\top}}
    -\ev{Q\wedge(\nabla^{M}Y)^{\top}}-\ev{Q\wedge A^{Y^{\perp}}}.
  \end{equation*}
  Since $Y$ is conformal, $(\nabla^{M}Y)^{\top}$ commutes with
  the complex structure $J_x$ of $\Sigma$ while
  $J_x^{*}Q=-Q$ so that $\ev{Q\wedge(\nabla^{M}Y)^{\top}}=0$.
  Thus, for any immersion $x$,
  \begin{equation}\label{eq:9}
    \d\eta_Y=\ev{\d^{\nabla^{\Sigma}}Q,Y^{\top}}
    -\ev{Q\wedge A^{Y^{\perp}}}.
  \end{equation}
  However, the holomorphicity of $q$ is equivalent to
  $\d^{\nabla^{\Sigma}}Q=0$ while the second summand
  vanishes by the isothermic condition \eqref{eq:6}.
\end{proof}
\begin{rem}
  When the conformal group acts transitively so that
  evaluation $\conf_M\to T_pM$ surjects for all $p\in M$, we
  see from \eqref{eq:9} that the converse also holds: if
  $\d\eta^q_x=0$ then $(x,q)$ is globally isothermic.
\end{rem}

This theory comes alive when we take $M=S^{n}$. Now
$\conf_{S^n}\cong\fso(n+1,1)$ is a simple Lie algebra and we
may use the Killing form to identify $\conf_{S^{n}}^{*}$
with $\conf_{S^n}$.  We therefore view $\eta_x^q$ as a
$\conf_{S^n}$-valued $1$-form which, in view of a remark
above, takes values in $\stab_{\conf_{S^{n}}}(x)^{\perp}\leq\conf_{S^n}$.
Moreover, for any $p\in S^n$,
$\stab_{\conf_{S^{n}}}(p)^{\perp}$ is abelian (it comprises
the infinitesimal translations of
$\R^n=S^n\setminus\set{p}$) so that $\eta_{x}^q$ takes
values in a bundle of abelian subalgebras.   From this, we
learn two things:
\begin{compactenum}
\item \emph{$(x,q)\colon\Sigma\to S^n$ is globally isothermic if and only if
  $\d+t\eta_x^q$ is a flat connection for each $t\in\R$}:
  \begin{equation*}
    R^{\d+t\eta_x^{q}}=t\d\eta_x^q+\half t^2[\eta_x^{q}\wedge\eta_x^{q}]=t\d\eta_x^q.
  \end{equation*}
  This gives an entry point to the well-developed integrable
  systems theory of isothermic surfaces, see
  \cite{BurHerPedPin97,Bur06,BurDonPedPin11,Cal03,CieGolSym95,Sch01},
  among many others, building on the classical results
  \cite{Bia05,Bia05a,Bou62a,Cal03,Dar99,Dar99b,Dar99e}, for
  more on this.
\item The key structure of $S^n$ that was needed for the
  final part of this analysis amounts to the requirement of
  a transitive action of a semisimple Lie group with
  parabolic infinitesimal stabilisers which have abelian
  nilradicals. Manifolds of this kind are the symmetric
  $R$-spaces and a surprising amount of the theory of
  isothermic surfaces goes through in this setting \cite{BurDonPedPin11}.
\end{compactenum}

\section{Conservation laws again}
\label{sec:cons-laws-again}

Let $x\colon\Sigma\to M$ be a CMC surface in a $3$-dimensional
space form.  Then $M$ can be conformally embedded as an open
subset of $S^{3}$ and, by Liouville's theorem,
$\conf_M=\conf_{S^3}$ so that, in particular,
$\g_M\leq\conf_{S^3}$.  We want to compare the moment class
\eqref{eq:21} $\mu_{x}\in H^1(\Sigma,\R)\tens\g_M^{*}$ with
the cohomology class $[\eta_{x}]$ of
$\eta_x:=\eta_x^{\II^{2,0}}$ which lies in
$H^1(\Sigma,\R)\tens\conf_{S^3}$.

Our main result is the following sequence of
observations:
\begin{compactenum}
\item There is a $G_M$-invariant symmetric bilinear form on
  $\g_M$ of signature $(3,3)$ and so a $G_M$-equivariant
  isomorphism $S\colon\g_M\to\g_M^{*}$.
\item For CMC $x\colon\Sigma\to M$, we have $[\eta_x]\in
  H^1(\Sigma,\R)\tens\g_{M}$.
\item $\mu_x=S[\eta_x]$.
\end{compactenum}
In this way, we arrive at an explicit representative of
$\mu_{x}$.

For this, we start by taking a unified approach to space
forms by viewing them as conic sections.

\subsection{Conic sections in general}
\label{sec:conic-sections}

Contemplate $\lor[n+1,1]$, an $(n+2)$-dimensional real
vector space with a symmetric bilinear form $\ip{\,,\,}$ of
signature $(n+1,1)$.  Denote by $\cL$ the lightcone of
$\lor[n+1,1]$.

Let $\g=\fso(n+1,1)$ which we shall frequently (and silently)
identify with $\Wedge^2\lor[n+1,1]$ via:
\begin{equation*}
  (a\wedge b)c=\ip{a,c}b-\ip{b,c}a,
\end{equation*}
for $a,b,c\in\lor[n+1,1]$.

Let $\q\in\lor[n+1,1]$ be non-zero, define the affine
hyperplane
\begin{equation*}
  E_{\q}:=\set{v\in\lor[n+1,1]\st \ip{v,\q}=-1}
\end{equation*}
and then the conic section $M_{\q}:=\cL\cap E_{\q}$.  Equip $M_{\q}$
with the (positive definite) metric induced from
$\lor[n+1,1]$.  Then $M_{\q}$ is an $n$-dimensional space
form\footnote{Strictly, when $\ip{\q,\q}>0$, $M_q$ has two
  components (it is a hyperboloid of two sheets) each of
  which is a space form.}
of curvature $-\ip{\q,\q}$ (see, for example,
\cite[\S1.1.4]{Her03}).

For $p\in M_{\q}$, the normal bundle to $M_{\q}$ in
$\lor[n+1,1]$ is spanned by $p,\q$ so that
$T_pM_{\q}=\Span{p,\q}^{\perp}\leq\lor[n+1,1]$.  We realise
$\g$ as the conformal vector fields $\conf_{M_{\q}}$ of $M_{\q}$ by
\begin{equation*}
  \Y(p)=-Yp-\ip{Yp,\q}p
\end{equation*}
where the minus sign ensures that $Y\mapsto (p\mapsto
\Y(p))\colon\g\cong\conf_{M_{\q}}$ is an isomorphism of Lie
algebras.  Note that this isomorphism identifies $\g_{M_{\q}}$
with $\stab(\q)=\Wedge^2\q^{\perp}\leq\g$.

We equip $\g$ with the invariant symmetric bilinear
form\footnote{This is a non-negative multiple of the Killing
  form of $\g$.}
$\B{\,,\,}$ given by
\begin{equation*}
  \B{Y_1,Y_2}=\half\trace Y_1Y_2
\end{equation*}
and remark that a short calculation gives
\begin{equation}
  \label{eq:10}
  \ip{\Y(p),v}=\B{Y,p\wedge v},
\end{equation}
for all $v\in T_pM_{\q}$.

With this in hand, let $(x,q)\colon\Sigma\to M_{\q}$ be an
isothermic surface and use $\I$ to define
$Q^{\#}\in\sect\Sym_{0}T\Sigma$ by
\begin{equation*}
  Q(U,V)=I(Q^{\#}U,V),
\end{equation*}
where, as usual, $Q=2\Re q$.  Then, for $\Y\in\conf_{M_{\q}}$,
\begin{equation*}
  \eta_Y(U)=Q(\Y^{\top},U)=\I(\Y^{\top},Q^{\#}U)=
  \ip{\Y\circ x,\d x\circ Q^{\#}U}.
\end{equation*}
From \eqref{eq:10}, this gives:
\begin{equation*}
  \eta^q_x(\Y)=\B{Y,x\wedge\d x\circ Q^{\#}}
\end{equation*}
so that, identifying $\g$ with $\g^{*}$ via $\B{\,,\,}$, we
conclude:
\begin{equation}
  \label{eq:12}
  \eta_x^q=x\wedge \d x\circ Q^{\#}\in\Omega^1_{\Sigma}\tens\g,
\end{equation}
(c.f.\ \cite[Equation~(3)]{BurSan12}).

Finally, for future use, we define a family of
$\g_{M_{\q}}$-valued $1$-forms on $M_{\q}$ as follows: fix
$\fo\in E_{\q}$, let $i\colon M_{\q}\to\lor[n+1,1]$ be the
inclusion and set $\alpha^{\fo}=(i-\fo)\wedge\d i$.  Note
that $i-\fo$ and so $\d i$ take values in $\q^{\perp}$ so
that $\alpha^{\fo}$ takes values in
$\Wedge^2\q^{\perp}=\g_{M_{\q}}$ as required.  Explicitly,
\begin{equation}
  \label{eq:13}
  \alpha^{\fo}_p(v):=(p-\fo)\wedge v,
\end{equation}
for $v\in T_pM_{\q}$.  We have
\begin{lem}
  \label{th:5}
  With $\alpha^{\fo}$ defined as above:
  \begin{compactenum}
  \item $\d\alpha^{\fo}(V,W)=2V\wedge W$, for $V,W\in
    \sect T\Sigma$.
  \item Modulo exact forms, $\alpha^{\fo}$ is independent of
    $\fo\in E_{\q}$.
  \end{compactenum}
\end{lem}
\begin{proof}
  By the Leibniz rule,
  \begin{equation*}
    \d\alpha^{\fo}=\d i\curlywedge\d i
  \end{equation*}
  where $\curlywedge$ is wedge product of
  $\lor[n+1,1]$-valued $1$-forms using the exterior product
  of $\lor[n+1,1]$ to multiply coefficients to yield a
  $\Wedge^2\lor[n+1,1]$-valued $2$-form.  Since we are
  already identifying $TM_{\q}$ with $\d i(TM_{\q})$, this reads
  \begin{equation*}
    \d\alpha^{\fo}(V,W)=\d i(V)\wedge\d i(W)-\d i(W)\wedge\d i(V)=
    2V\wedge W,
  \end{equation*}
  which settles Item (1).

  For Item (2), note that for $v\in\q^{\perp}$,
  \begin{equation*}
    \alpha^{\fo+v}=\alpha^{\fo}+\d((i-\fo)\wedge v).
  \end{equation*}
\end{proof}

\subsection{$3$-dimensional conic sections}
\label{sec:3-dimensional-conic}

Let us focus our attention on the case $n=3$ of principal
interest to us.  Here there is a new ingredient:
$\dim \q^{\perp}=4$ so that wedge product
$\Wedge^2\q^{\perp}\times\Wedge^2\q^{\perp}\to
\Wedge^4\q^{\perp}\cong\R$ is a symmetric bilinear form of
signature $(3,3)$ and so induces an isomorphism
$S\colon\g_{M_{\q}}\to\g_{M_{\q}}^{*}$.

In more detail, fix a volume form $\det\in(\Wedge^5\lor)^{*}$ and
make the following:
\begin{defn}
  Let $S\colon\Wedge^2\q^{\perp}\to(\Wedge^2\q^{\perp})^{*}$ be
  given by
  \begin{equation}
    \label{eq:14}
    S(Y_1)(Y_2)=\det(\fo\wedge Y_1\wedge Y_2),
  \end{equation}
  for $Y_1,Y_2\in\Wedge^2\q^{\perp}$.  Here $\fo$ is any
  element of $E_{\q}$ and $S$ is independent of that choice.
\end{defn}
\goodbreak
\begin{rems}
\item[]
  \begin{compactenum}
  \item Note that $S$ is equivariant for the action of
    $\rSL(\q^{\perp})$ and so, in particular, viewed as an
    isomorphism $\g_{M_{\q}}\cong\g_{M_{\q}}^{*}$,
    \emph{$S$ intertwines the adjoint and co-adjoint actions of
      $G_{M_{\q}}$}.
  \item When $K=-\ip{\q,\q}\neq 0$, $\q^{\perp}$ is
    isomorphic to $\R^4$ or $\R^{3,1}$ according to the sign
    of $K$.  Now $\g_{M_{\q}}$ is semisimple and we can
    equivariantly identify $\g_{M_{\q}}$ with its dual via
    (minus) the Killing form.  In this setting, we see from
    \eqref{eq:14} that
    $S\colon\g_{M_{\q}}\to\g_{M_{\q}}$ is essentially the Hodge
    star operator.
  \item When $K=0$, $\g_{M_{\q}}$ is the Euclidean Lie
    algebra and far from semisimple: it is the semi-direct
    product $\fe(3)=\fso(3)\oplus\R^{3}$.  In general, the
    adjoint and co-adjoint representations of $\fe(n)$ are
    very different and the
    existence of $S$ when $n=3$ came as a surprise to us.
  \end{compactenum}
\end{rems}

For $p\in M_{\q}$, orient $N_pM_{\q}$ so that $p\wedge \q>0$
whence $\vol_{M_{\q}}\restr{p}=i_{p\wedge\q}\det$.  We now
have:
\begin{lem}
  \label{th:7}
  Let $v,w\in T_pM_{\q}$ and $Y\in\g_{M_{\q}}$.  Then
  \begin{equation}
    \label{eq:11}
    S(v\wedge w)(Y)=i_{\Y(p)}\vol_{M_{\q}}(v,w).
  \end{equation}
\end{lem}
\begin{proof}
  Unravelling the definitions, we see that \eqref{eq:11}
  amounts to
  \begin{equation}
    \label{eq:15}
    \fo\wedge v\wedge w\wedge Y = -p\wedge\q\wedge Yp\wedge
    v\wedge w.
  \end{equation}
  It suffices to prove this for decomposable
  $Y\in\Wedge^{2}\q^{\perp}$ so let $Y=a\wedge b$, for
  $a,b\in\q^{\perp}$.  Then
  \begin{equation*}
    a=a^{\top}-\ip{a,p}(\q+\ip{\q,\q}p),
  \end{equation*}
  with $a^{\top}\in T_{p}M_{\q}$, and similarly for $b$ so that
  \begin{align*}
    \fo\wedge v\wedge w\wedge Y&=p\wedge v\wedge w\wedge
    Y=p\wedge v\wedge w\wedge a\wedge b\\
    &=-p\wedge v\wedge w\wedge a\wedge \ip{b,p}\q-p\wedge
    v\wedge w\wedge\ip{a,p}\q\wedge b\\
    &=-p\wedge\q\wedge v\wedge
    w\wedge\bigl(\ip{a,p}b-\ip{b,p}a\bigr)\\
    &=-p\wedge\q\wedge Yp\wedge v\wedge
    w,
  \end{align*}
  as required.  Here we used $\Wedge^5\q^{\perp}=0$ in the
  first line to replace $\fo$ with $p$ and $\Wedge^4T_pM_{\q}=0$ in the
  second line to replace $a$ or $b$ with its component along $\q$.
\end{proof}

As a corollary, we find an explicit representative for
$\alpha_{Y}$:
\begin{cor}
  \label{th:8}
  For $Y\in\g_{M_{\q}}$, we have
  \begin{equation}
    \label{eq:16}
    S(\alpha^{\fo})(Y)\equiv 2\alpha_Y\mod \d\Omega^0_{M_{\q}}.
  \end{equation}
\end{cor}
\begin{proof}
  We must show that
  \begin{equation*}
    \d \bigl(S(\alpha^{\fo})(Y)\bigr)=2i_{\Y}\vol_{M_{\q}}.
  \end{equation*}
  However, for $V,W\in\sect TM_{\q}$,
  \begin{equation*}
    \d
    \bigl(S(\alpha^{\fo})(Y)\bigr)(V,W)=S(\d\alpha^{\fo}(V,W))(Y)=2S(V\wedge
    W)(Y)=2i_{\Y}\vol_{M_{\q}}(V,W),
  \end{equation*}
  where we have used \cref{th:5} for the second equality and
  \cref{th:7} for the last.
\end{proof}

\subsection{CMC surfaces in space forms}
\label{sec:cmc-surfaces-space}

Let $x\colon\Sigma\to M_{\q}$ be a CMC surface in a
$3$-dimensional space form.  We have seen that $x$ is
isothermic with $q=\II^{2,0}$ so that, from \eqref{eq:12},
\begin{equation*}
  \eta^q_{x}=x\wedge \d x\circ A^{N}_{0}.
\end{equation*}
Moreover, $\d x\circ
A^{N}_{0}=-\d(N+Hx)\in\Omega^1_{\Sigma}\tens \q^{\perp}$
from which we learn two things.  First,
\begin{equation*}
  \eta^q_x\q=d(N+Hx)
\end{equation*}
is exact so that $[\eta_x^q]\in
H^1(\Sigma,\R)\tens\g_{M_{\q}}$.  Second, we have
\begin{equation*}
  \eta_x^q\equiv \d x\wedge (N+Hx)\mod \d\Omega^0_{\Sigma}\tens\g_{M_{\q}}.
\end{equation*}
We are now in a position to prove our main result:
$S[\eta_x^q]=\mu_x$.  First note that, by \cref{th:7}, for
$Y\in\g_{M_{\q}}$,
\begin{equation*}
  S(\d x\wedge N)(Y)=\vol_{M_{\q}}(\Y\circ x,\d
  x,N)=\vol_{M_{\q}}(\Y^{\top},\d x,N)=i_{Y^{\top}}\vol_x.
\end{equation*}
On the other hand, notice that
$x^{*}\alpha^{\fo}=(x-\fo)\wedge \d x$ so that, working modulo exact forms,
\begin{equation*}
  S(\d x\wedge x)(Y)\equiv S(\d x\wedge (x-\fo))(Y)=
  -S(x^{*}\alpha^{\fo})(Y)\equiv -2x^{*}\alpha_Y,
\end{equation*}
by \cref{th:8}.  Putting this all together, we get
\begin{equation*}
  S[\eta_x^q](Y)=[i_{\Y^{\top}}\vol_x-2Hx^{*}\alpha_{Y}]=\mu_x(Y).
\end{equation*}
To summarise:
\begin{thm}
  \label{th:9}
  Let $x\colon\Sigma\to M_{\q}$ be a CMC surface in a
  $3$-dimensional space form with retraction form $\eta_x^q$.
  Then
  \begin{compactenum}
  \item $[\eta_x^q]\in H^1(\Sigma,\R)\tens\g_{M_{\q}}$.
  \item $S[\eta_x^q]=\mu_{x}$.
  \end{compactenum}
\end{thm}

\subsection{Application: parallel CMC surfaces}
\label{sec:appl-parall-cmc}

Let $x\colon\Sigma\to M$ have constant mean curvature $H$ in a
$3$-dimensional space form $M$ of sectional curvature $K$.
Then, when $H^2+K>0$, and away from umbilics of $x$, one can
find a parallel CMC surface $\dual{x}\colon\Sigma\to M$ for which
the mean curvature, induced conformal structure and Hopf
differential coincide with those of $x$.  When $M=\R^3$,
this is a classical result of Bonnet (c.f\
\cite[p.~266]{bo60}) and $\dual{x}=x+N/H$.  For $K\neq 0$,
the parallel surfaces come in congruent pairs and are
described in Palais--Terng \cite{PalTer88} or
\citelist{\cite{HerTjaZue97}*{Part II, \S5}\cite{Her03}*{\S2.7}}.

We will give a self-contained account of this result and use it to
reveal a symmetry of the moment class:
$\mu_x=\mu_{\dual{x}}$.

So fix $\q$ with $\ip{\q,\q}=-K$ and let
$x\colon\Sigma\to M_{\q}$ be an immersion with constant mean
curvature $H$.  Let $m\in\R$ solve
\begin{equation}
  \label{eq:17}
  m^{2}-2Hm-K=0,\,\text{equivalently,}\,(m-H)^2=H^2+K.
\end{equation}
That $m$ be real is equivalent to $H^2+K\geq 0$ and we
assume henceforth that $m\neq 0,H$, so that, in particular,
$H^2+K>0$.

We now define $\dual{x}\colon\Sigma\to\lor$ by
\begin{equation}
  \label{eq:18}
  \dual{x}=\tfrac{1}{m-H}(N+Hx+\q/m)
\end{equation}
and note that, as a consequence of \eqref{eq:17}, we have
\begin{equation*}
  \ip{\xh,\xh}=0,\qquad \ip{\xh,\q}=-1
\end{equation*}
so that $\xh\colon\Sigma\to M_{\q}$.  We have:
\begin{prop}
  \label{th:10}
  $\xh\colon\Sigma\to M_{\q}$ is parallel to $x$ and immerses away from the
  umbilics of $x$.  Where $\xh$ immerses:
  \begin{compactenum}
  \item \label{item:3}the metric $\dual\I$ induced by $\xh$ is conformal
    to $\I$;
  \item \label{item:4}$\xh$ has (oriented) normal $\dual{N}$ given by
    \begin{equation}
      \label{eq:19}
      \dual{N}=\tfrac{1}{m-H}(Kx-HN-\q);
    \end{equation}
  \item \label{item:5}$\xh$ has constant mean curvature $H$;
  \item \label{item:6}the Hopf differentials of $x$ and $\dual{x}$
    coincide: $q=\dual{q}$.
  \end{compactenum}
\end{prop}

Before proving this, we pause to extract an interesting
result.  For any CMC surface $x$, we have already seen that
\begin{equation*}
  \eta^q_{x}=\d(N+Hx)\wedge x.
\end{equation*}
In the case at hand, \eqref{eq:18} gives
\begin{equation*}
  \eta_x^q=(m-H)\d\xh\wedge x.
\end{equation*}
On the other hand, \eqref{eq:18} and \eqref{eq:19}, along
with \eqref{eq:17}, yield
\begin{equation*}
  \dual{N}+H\dual{x}=\tfrac{1}{m-H}((H^2+K)x+(1/m-1)\q)=(m-H)x+\tfrac{1/m-1}{m-H}\q.
\end{equation*}
Thus matters are completely symmetric in $x$ and $\xh$:
\begin{equation*}
  \eta_{\xh}^{\dual{q}}=(m-H)\d x\wedge\xh
\end{equation*}
so that we see
\begin{equation*}
  \eta_x^q=\eta_{\dual{x}}^{\dual{q}}+(m-H)\d(\xh\wedge x).
\end{equation*}
The two retraction forms are therefore cohomologous and so,
in view of \cref{th:9}, we conclude:
\begin{thm}
  \label{th:11}
  The moment classes of a CMC surface $x$ and its parallel
  surface $\xh$ coincide: $\mu_x=\mu_{\xh}$.
\end{thm}

\begin{rems}[for the experts]
\item{}
  \begin{compactenum}
  \item CMC surfaces in space forms and their parallel
    surfaces are a particular case of a more general theory
    of \emph{special isothermic surfaces} developed by
    Bianchi \cite{Bia05} and Darboux \cite{Dar99}.  In this
    setting, the parallel CMC surfaces arise as
    \emph{complementary isothermic surfaces}.  See
    \cite{BurSan12} for a modern perspective on this.
  \item In view of the above remark, that the two retraction
    forms are cohomologous can be also understood from a
    more general perspective: in fact, a complementary
    surface is a special kind of Darboux transform and one
    can show \cite{MR3523116} that \emph{any} Darboux pair
    $x,\dual{x}\colon\Sigma\to S^n$ have cohomologous retraction
    forms.
  \end{compactenum}
\end{rems}

Now let us take care of unfinished business:
\begin{proof}[Proof of \cref{th:10}]
  That $\xh$ is a parallel surface to $x$ in $M_{\q}$
  follows from the fact that it is a constant coefficient
  linear combination of $x$, $N$ and $\q$. 

 Next, differentiate \eqref{eq:18} to see that
 \begin{equation}\label{eq:20}
   \d\xh=-\tfrac{1}{m-H}\d x\circ A_0^N
 \end{equation}
 and recall that $A_0^{N}$ is conformal ($(A^{N}_0)^2=-\det
 A^{N}_{0}$) and so bijects
 except at its zeros which are the umbilic points of $x$.
 Thus $\xh$ immerses off the umbilic locus of $x$ and, moreover,
 $\dual{\I}=-\det A^N_0\I$ there, settling Item
 \ref{item:3}.  For the rest of the proof, we restrict
 attention to the complement of the umbilic locus of $x$ so
 that $\xh$ immerses.

 For Item \ref{item:4} , one readily checks, using \eqref{eq:17} that
 $\dual{N}$ defined by \eqref{eq:19} has unit length and is
 orthogonal to $\xh,\d\xh(T\Sigma)=\d x(T\Sigma),\q$.
 Moreover, \eqref{eq:17} and \eqref{eq:20} yield
 \begin{equation*}
   \xh^{*}(i_{\dual N}\vol_{M\q})=-\tfrac{\det A^N_0}{(m-H)^{2}}\vol_x
 \end{equation*}
 which is a positive multiple of $\vol_x$ so that $\dual{N}$
 is the correctly oriented unit normal to $\xh$.

 For the rest, we compute the shape operator $\dual{A}^{\dual N}$
 of $\xh$:
 \begin{align*}
   \d\dual{N}&=\tfrac{1}{m-H}\d(Kx-HN)=\tfrac{1}{m-H}\d
   x(K+HA^N)\\
   &=\tfrac{1}{m-H}\d x(H^2+K+HA^N_0)\\
   &=-\d\xh\bigl((A^N_0)^{-1}(H^2+K+HA^N_0)\bigr)=-\d\xh\bigl(H+(H^2+K)(A_0^N)^{-1}\bigr),
 \end{align*}
 where we used \eqref{eq:20} to reach the last line.  Since
 $(A_0^N)^{-1}$ is trace-free, we immediately conclude
 \begin{equation*}
   \dual{H}=H,\qquad \dual{A}_0^{\dual{N}}=(H^2+K)(A_0^N)^{-1},
 \end{equation*}
 settling Item \ref{item:5}.  Further, from \eqref{eq:17} and
 \eqref{eq:20}, we now get
 \begin{equation*}
   \dual\II_0=\dual{\I}(\dual{A}^{\dual{N}}_0,-)=\tfrac{H^2+K}{(m-H)^2}\I(-,A^N_0)=\II_0
 \end{equation*}
 and so we are done.
\end{proof}

\begin{rem}
  When $K\neq 0$, there are two non-zero solutions $m_{\pm}$
  to \eqref{eq:17} with $m_+-H=H-m_-$.  It is an easy
  exercise to see that the corresponding parallel surfaces
  $\xh_{\pm}$ are related by minus the reflection in the
  hyperplane orthogonal to $\q$.  Geometrically, when $K>0$,
  this means that $\xh_{\pm}$ are antipodal while, when
  $K<0$, $\xh_{\pm}$ lie in different sheets of the
  double-sheeted hyperboloid $M_{\q}$.   In the latter case,
  choose $m$ with $\abs{m}$ maximal to get $\xh$ in the same
  sheet as $x$.
\end{rem}

\subsection{Flux, torque and generalisations}
\label{sec:flux-torq-gener}

We conclude this section by using our results to write down
explicit representatives of the moment class in classical
terms.  When $K=0$, this will recover the flux form of Meeks
et al.\ \cite{MR3525098} and provide an analogous formula
for the torque form.  Moreover, we offer, apparently novel,
generalisations of these to the case $K\neq 0$.

First we take $K=0$ and take 
$\fo\in M_{\q}$.  We identify
$x_0\colon\Sigma\to\R^3=\Span{\fo,\q}^{\perp}$ with
$x\colon\Sigma\to M_{\q}$ via
\begin{equation*}
  x=\fo+x_0+\half\ip{x_0,x_0}\q
\end{equation*}
and unit normals to these by
\begin{equation*}
  N=N_0+\ip{x_0,N_0}\q.
\end{equation*}
When $x$ has constant mean curvature $H$, we have
\begin{equation}\label{eq:22}
  \eta_x^q\equiv \d x_0\wedge(N_0+Hx_0)+\bigl((N_0\wedge\d
  x_0)x_0+H\ip{x_0,x_0}\d x_{0}\bigr)\wedge\q,
\end{equation}
modulo exact terms.  Now, for $v\in\R^{3}$, let $Y\in\fe(3)$ be given by
$Y=\q\wedge v$ so that $Y$ is the constant vector field
$Y(p)=v$ on $\R^{3}$.  Then \eqref{eq:22} gives
\begin{equation*}
  S(\eta^q_x)(Y)\equiv\det(\fo\wedge\q\wedge\eta^q_x\wedge
  v)=
  \vol_{\R^3}(\d x_0,N_0+Hx_{0},v)=\ip{\d x_0\times
    (N_0+Hx_0),v}.
\end{equation*}
Thus, identifying the constant vector fields in $\fe(3)$ with
$\R^3$ via evaluation and using the metric to identify
$\R^3$ with its dual, we conclude
\begin{prop}
  \label{th:12}
  Let $x_0\colon\Sigma\to\R^3$ have constant mean curvature $H$.
  The restriction of the moment class of $x_0$ to the constant vector
  fields is represented by the \emph{flux form} $\d
  x_0\times(N_{0}+H x_0)\in\Omega^1_{\Sigma}\tens\R^3$:
  \begin{equation*}
    \mu_{x_0}\restr{\R^3}=[\d
  x_0\times(N_{0}+H x_0)].
  \end{equation*}
\end{prop}

Similarly, the infinitesimal rotations around zero are
$\fso(3)=\Wedge^2\R^{3}\leq\Wedge^2\q^{\perp}=\fe(3)$.  For
$Y=v\wedge w\in\fso(3)$, we have
\begin{align*}
  S(\eta^q_x)(Y)&\equiv\det(\fo\wedge\eta^q_x\wedge v\wedge
  w)=
  -\vol_{\R^3}(\bigl((N_0\wedge\d
  x_0)x_0+H\ip{x_0,x_0}\d x_{0}\bigr),v,w)\\
  &=
  -\ip{\bigl((N_0\wedge\d
  x_0)x_0+H\ip{x_0,x_0}\d x_{0}\bigr),v\times w}.
\end{align*}
Thus, identifying $\fso(3)^{*}\cong\R^3$ via $\ev{u,v\wedge
  w}=\ip{u,v\times w}$, we have
\begin{prop}
  \label{th:13}
  Let $x_0\colon\Sigma\to\R^3$ have constant mean curvature $H$.
  The restriction of the moment class to $\fso(3)$ is
  represented by the \emph{torque form} $-(N_0\wedge\d
  x_0)x_0-H\ip{x_0,x_0}\d
  x_{0}\in\Omega^1_{\Sigma}\tens\R^{3}$:
  \begin{equation*}
     \mu_{x_0}\restr{\fso(3)}=[-(N_0\wedge\d
  x_0)x_0-H\ip{x_0,x_0}\d x_{0}].
  \end{equation*}
\end{prop}

When $K\neq 0$, we get more uniform statements.  Firstly,
$\g_{M_{\q}}=\Wedge^2\q^{\perp}$ is identified with its dual
using the metric\footnote{This is the negative of $\ip{\,,\,}_{\g}$
discussed above.} on $\q^{\perp}$ induced from that of
$\q^{\perp}$.  Secondly, take $\fo=\q/K$ and, for any
$x\colon\Sigma\to M_{\q}$, write
\begin{equation*}
  x=x_0+\fo
\end{equation*}
to get $x_0\colon\Sigma\to\q^{\perp}$ taking values in a
(pseudo-)sphere: $\ip{x_0,x_0}=1/K$ with normal $N_{0}=N$
(since $N$ is already orthogonal to $\fo$ in this case).
Then, when $x_0$ has constant mean curvature $H$, starting
from
\begin{equation*}
  \eta_x^q=(m-H)\d x_{0}\wedge \xh_0, 
\end{equation*}
we conclude after a computation:
\begin{prop}
  \label{th:14}
  Let $x_0$ have constant mean curvature $H$ in a
  (pseudo-)sphere of curvature $K$ in $\R^{p,q}$, where
  $(p,q)=(4,0)$ or $(3,1)$ according to the sign of
  $K$. Then the moment class of $x_0$ is represented by the
  \emph{moment form}
  $(Kx_0-HN)\wedge *\d
  x_0\in\Omega^1_{\Sigma}\tens\Wedge^2\R^{p,q}$:
  \begin{equation*}
    \mu_{x_0}=[(Kx_0-HN)\wedge *\d
    x_0],
  \end{equation*}
  where $*$ is the Hodge star operator of $\Sigma$.
\end{prop}

\section{Noether's theorem}
\label{sec:noethers-theorem}

This section attempts to partially resolve the somewhat
mysterious relationship between the moment class
\begin{equation*}
  \mu_x\in H^{k-1}(\Sigma,\R)\otimes\g_M^*
\end{equation*}
for a constant mean curvature hypersurface
$x\colon \Sigma^k\to M^{k+1}$ in a Riemannian manifold in the case
$k=2$ and the retraction class
\begin{equation*} [\eta^q_x]\in H^{1}(\Sigma,\R)\otimes \frak{c}_{M}^*
\end{equation*}
for a globally isothermic surface $(x,q)\colon\Sigma^2\to M^n$
into a conformal manifold.  We shall show that both classes
have a variational origin as Noether currents.

We begin with a brief description of Noether analysis suitable
for our discussion based on \cite{MR1701599}. Let
$\Sigma^k$, $M^n$ be manifolds and
$\mathcal{F}\subset C^{\infty}(\Sigma,M)$ an open subset
which, in our examples, will be the set of immersions
$\Sigma\to M$.
\begin{defn}
  A {\em Lagrangian} is a smooth map
  \begin{equation*}
    \mathcal{L}\colon \mathcal{F}\to W,
  \end{equation*}
  where $W$ is a manifold.
\end{defn}
In this context one is interested in the {\em critical
  locus} of $\mathcal{L}$, that is, maps $x\in\mathcal{F}$
for which $\d^{\mathcal{F}}_{x}\mathcal{L}$ is not
surjective.

In many but not all cases, Lagrangians
$\mathcal{L}\colon \mathcal{F}\to \R$ factorise
\begin{equation*}
  \mathcal{L}(x)=\int_{\Sigma} \ell_x
\end{equation*}
via a {\em Lagrangian density}
$\ell\colon \mathcal{F}\to \Omega^{k}_{\Sigma}$. In this
setting, $x\in\mathcal{F}$ is critical if and only
if
\begin{equation*}
  \d^{\mathcal{F}}_{x}\ell(\dot{x})=0,
\end{equation*}
for all tangent vectors (variations)
$\dot{x}\in T_x\mathcal{F}=\sect{x^*TM}$.
Thus, the critical points of $\mathcal{L}$ are determined by
the $\Omega^k_{\Sigma}$-valued $1$-form
\begin{equation*}
  \d^{\mathcal{F}}\ell\in
  \Omega^1_{\mathcal{F}}(\Omega^k_{\Sigma})
\end{equation*}
which motivates
\begin{defn}
  A {\em variational operator} is an
  $\Omega^k_{\Sigma}$-valued $1$-form on $\mathcal{F}$, that
  is,
  \begin{equation*}
    \Dif\in \Omega^1_{\mathcal{F}}(\Omega^k_{\Sigma}).
  \end{equation*}
  Thus, for each $x\in\mathcal{F}$ we have a linear map
  \begin{equation*}
    \Dif_x\colon \sect{x^*TM}\to \Omega^k_{\Sigma}
  \end{equation*}
  which we assume to be a differential operator. We call
  $x\in\mathcal{F}$ a {\em critical point} of $\Dif$ if
  \begin{equation*}
    \Dif_x(V)=0
  \end{equation*}
  for all $V\in \sect{x^*TM}$.
\end{defn}
We note from the discussion above that, for factorisable
Lagrangians $\mathcal{L}$, the variational operator
$\Dif$ is the derivative $\d^{\mathcal{F}}\ell$ of the Lagrangian
density $\ell$.
\begin{lem}[c.f.\ \cite{MR1701599}*{\S2.4}]\label{lem:EL}
  Let $\Dif\in \Omega^1_{\mathcal{F}}(\Omega^k_{\Sigma})$ be a
  variational operator for which $\Dif_x$ is first order, for
  every $x\in\mathcal{F}$. Then there exist unique $1$-forms
  \begin{equation*}
    \beta\in\Omega^1_{\mathcal{F}}(
    \Omega^{k-1}_{\Sigma})\,\quad\text{and}\quad
    E\in\Omega^1_{\mathcal{F}}(\Omega^{k}_{\Sigma})
  \end{equation*}
  such that for $x\in\mathcal{F}$ we have
  \begin{compactitem}
  \item
    $\beta_x\colon \sect{x^*TM}\to \Omega^{k-1}_{\Sigma}$
    and $E_x\colon \sect{x^*TM}\to \Omega^{k}_{\Sigma}$ are
    $C^{\infty}(\Sigma,\R)$-linear and so tensorial.
  \item For all $V\in \sect{x^*TM}$,
    \begin{equation}\label{eq:23}
      \Dif_x(V)=\d(\beta_x(V))+E_x(V).
    \end{equation}
  \end{compactitem}
  The $1$-form
  $E\in \Omega^1(\mathcal{F},\Omega^{k}_{\Sigma})$ is called
  the {\em Euler-Lagrange operator} and the $1$-form
  $\beta\in\Omega^1_{\mathcal{F}}( \Omega^{k-1}_{\Sigma})$
  the {\em variational $1$-form}.
\end{lem}
\begin{proof}
  Since $\Dif_x$ is a first order linear differential operator,
  we have
  \begin{equation*}
    \Dif_x (fV)=f\Dif_x V +S_{\d f}V
  \end{equation*}
  where the bundle homomorphism
  $S\colon T^*\Sigma\to\Hom(x^*TM, \wedge^k T^*\Sigma))$ is the
  symbol of $\Dif_x$. Contraction then gives a bundle
  homomorphism
  $\beta_x\colon x^*TM\to \wedge^{k-1} T^*\Sigma$. For
  $V\in\sect{x^*TM}$ one then easily checks that
  \begin{equation*}
    E_x(V):=\Dif_x V-\d(\beta_x(V))
  \end{equation*}
  is $C^{\infty}(\Sigma,\R)$-linear as required. Uniqueness
  follows since any other decomposition of $\Dif$ by
  $\tilde{\beta}$, $\tilde{E}$ gives
  $\d f\wedge (\beta_x(V)-\tilde{\beta}_x(V))=0$, for all
  $f\in C^{\infty}(\Sigma,\R)$ and $V\in\sect{x^*TM}$.
\end{proof}
We thus get a version of the Euler-Lagrange condition for
critical points in our setup:
\begin{cor}\label{cor:critical}
  $x\in \mathcal{F}$ is a critical point for the variational
  operator $\Dif$ if and only if
  \begin{equation*}
    E_x=0.
  \end{equation*}
\end{cor}
\begin{proof}
  Let $V\in \sect{x^*TM}$ be compactly supported. Then
  \eqref{eq:23} and Stokes' theorem give:
  \begin{equation*}
    \int_{\Sigma} \Dif_x(V)=\int_{\Sigma}
    \d(\beta_x(V))+\int_{\Sigma} E_x(V)=\int_{\Sigma} E_x(V).
  \end{equation*}
\end{proof}

Consider a Lagrangian $\mathcal{L}$ which factorises via a
Lagrangian density $\ell$ and assume we have a Lie group $G$
acting on $\mathcal{F}$ leaving $\ell$ (and thus also
$\mathcal{L})$ invariant, that is, $\ell_{gx}=\ell_x$ for
all $g\in G$. Differentiating with respect to $g$ we obtain
$\d^{\mathcal{F}}_{x}\ell(\mathcal{V})=0$ for the
fundamental vector fields
$\mathcal{V}\in\sect{T\mathcal{F}}$ of the group action.
\begin{defn}
  Let $\Dif\in\Omega^1_{\mathcal{F}}(\Omega^k_{\Sigma})$ be a
  variational operator. A vector field
  $\mathcal{V}\in\sect{T\mathcal{F}}$ is a {\em symmetry}
  of $\Dif$ if
  \begin{equation*}
    \Dif(\mathcal{V})=0.
  \end{equation*}
\end{defn}
\begin{thm}[Noether]
  If $\mathcal{V}\in\sect{T\mathcal{F}}$ is a symmetry of
  the variational operator $\Dif$ and $x\in \mathcal{F}$ a
  critical point of $\Dif$, then
  $\beta_x(\mathcal{V}_x)\in\Omega^{k-1}_{\Sigma}$ is
  closed.

  We call $\beta_x(\mathcal{V}_x)$ the {\em Noether current}
  of $\Dif$ for the symmetry $\mathcal{V}$.
\end{thm}
\begin{proof}
  Since $x$ is critical and $\mathcal{V}$ is a symmetry,
  both $E_x(\mathcal{V}_{x})$ and $\Dif_x(\mathcal{V}_{x})$
  vanish and so the result is immediate from \eqref{eq:23}.
\end{proof}

In the remainder of this section we apply our Noether
analysis to CMC hypersurfaces and globally isothermic
surfaces to show that the moment and retraction classes
arise as Noether currents.

\subsection{Constant mean curvature hypersurfaces}
Let $\mathcal{F}$ be the space of immersions
$x\colon \Sigma^k\to M^{n}$ into a Riemannian manifold. We
additionally assume
\begin{itemize}
\item $\Sigma=\partial \hat{\Sigma}$ is the boundary of a
  manifold $\hat{\Sigma}^{k+1}$;
\item $x$ extends to an immersion 
  $\hat{x}\colon\hat{\Sigma}\to M$ so that
  $\hat{x}_{|\Sigma}=x$.
\end{itemize}
Consider the Lagrangian
\begin{equation*}
  \mathcal{L}(x)=\int_{\Sigma} \vol_x +\lambda
  \int_{\hat{\Sigma}}\vol_{\hat{x}}
\end{equation*}
with a Lagrange multiplier $\lambda\in\R$. Critical points
of $\mathcal{L}$ extremise the induced volume on $\Sigma$
constrained by the induced enclosed volume on
$\hat{\Sigma}$.  For a variation
$V=\d x(V^{\top})+V^{\perp}\in\sect{x^*TM}$, we obtain
\begin{equation*}
  \d^{\mathcal{F}}_{x}\vol(V)=\d(i_{V^{\top}}\vol_x)-k(\mcv, V^{\perp})\vol_x
\end{equation*}
with $\mcv\in\sect{N\Sigma}$ the mean curvature vector
field.

Now suppose that  $n=k+1$, that is, $x$ is a
hypersurface immersion, and that
$\hat{V}\in\sect{\hat{x}^*TM}$ extends $V$.  Then
\begin{equation*}
  \d^{\mathcal{F}}_{\hat{x}}\vol(\hat{V})=\d(i_{\hat{V}}\vol_{\hat{x}})
\end{equation*}
since there is no normal component. We therefore arrive at
\begin{equation*}
  \d^{\mathcal{F}}_x\mathcal{L}(V)=\int_{\Sigma}
  \d(i_{V^{\top}}\vol_x)-k(\mcv,V^{\perp})\vol_x+\lambda(V^{\perp},N)\vol_x
\end{equation*}
where we used
$i_{\hat{V}}\vol_{\hat{x}_{|\Sigma}}=(V^{\perp},N)\vol_x$
for the unit normal field $N\in\sect{N\Sigma}$ of
$x\colon \Sigma\to M$.

We can now read off the ingredients of \cref{lem:EL}: the variational operator
$\Dif\in\Omega^1_{\mathcal{F}}(\Omega^k_{\Sigma})$ for $\mathcal{L}$ reads
\begin{equation}\label{eq:varop-cmc}
  \Dif_x(V)= \d(i_{V^{\top}}\vol_x)-k(\mcv,V^{\perp})\vol_x+\lambda(V^{\perp},N)\vol_x,
\end{equation}
for $V\in T\mathcal{F}_x=\sect{x^*TM}$, while
\begin{align*}
  \beta_x(V)&=i_{V^{\top}}\vol_x\\
  E_x(V)&=(-k(\mcv,V^{\perp})+\lambda(V^{\perp},N))\vol_x.
\end{align*}
In particular, $x\in\mathcal{F}$ is a critical point of
$\mathcal{L}$ if and only if
\begin{equation*}
  k\mcv=\lambda N\quad\text{equivalently}\quad \lambda=k
  H,
\end{equation*}
that is, $x$ has constant mean curvature $H$.

Now let $V=Y\circ x$ for a Killing field $Y$ on $M$.  Then
$V$ is not, in general, a symmetry of $\mathcal{L}$ since we
have formulated the volume as a boundary integral.  However,
$V$ \emph{is} a symmetry for the area functional so that, for
any immersion $x$ we have the identity \eqref{eq:1}
\begin{equation*}
  \d(i_{Y^{\top}}\vol_x)=k(Y\circ x,\mcv)\vol_x
\end{equation*}
of \cref{th:1}, which, as we have seen, yields the moment
class $\mu_x$ under appropriate conditions on the cohomology
of $M$.

\subsection{Isothermic surfaces}
Let $\Sigma^2$ be an oriented compact surface without
boundary.  We collect some basic facts about Teichm\"uller
space following Bohle--Peters--Pinkall \cite{BohPetPin08}.
Consider
\begin{equation*}
  \mathcal{T}=\set{J\in\sect{\End(T\Sigma)}\st
  J^2=-1}/\Diff_0(\Sigma),
\end{equation*}
the space of equivalence classes of complex structures $J$
on $\Sigma$ modulo diffeomorphisms isotopic to the identity.
Then:
\begin{itemize}
\item The tangent space at $[J]$ to $ \mathcal{T}$ is
  \begin{equation*}
    T_{[J]}\mathcal{T}=\set{T\in \sect{\End(T\Sigma})\st 
    TJ=-JT}/\set{L_X J\st X\in \sect{T\Sigma}}.
  \end{equation*}
\item The cotangent space at $[J]$ to $ \mathcal{T}$ is
  \begin{equation*}
    T^{*}_{[J]}\mathcal{T}=\set{\sigma\in\sect{\Hom(T\Sigma,T^*\Sigma)}\st 
    J^{t}\sigma=\sigma J,\, \d^{\nabla} \sigma=0}
  \end{equation*}
  where $\nabla$ is any torsion-free connection with $\nabla
  J=0$. As we remarked above, the map
  \begin{equation*}
    H^0(K_{\Sigma}^2)\to T^{*}_{[J]}\mathcal{T}\colon
    q\mapsto \sigma=2\Re q
  \end{equation*}
  is a linear isomorphism once we interpret
  $\sigma\in T^{*}_{[J]}\mathcal{T}$ as a trace-free
  symmetric bilinear form
  $\sigma\in\sect{S^2_{0} T^*\Sigma}$.
\item The $L^2$-pairing
  \begin{equation*}
    T^{*}_{[J]}\mathcal{T}\times T_{[J]}\mathcal{T}\to
    \R\colon (\sigma,[T])\mapsto
    \int_{\Sigma}\ev{\sigma\wedge T}
  \end{equation*}
  is well-defined and non-degenerate. We view $\sigma$ and
  $T$ as $T^{*}\Sigma$ respectively $T\Sigma$ valued $1$-forms
  and the wedge product is over the evaluation pairing
  $\ev{\,,\,}\colon T^*\Sigma\times T\Sigma\to \R$.
\end{itemize}

Now let $\mathcal{F}$ be the space of immersions
$x\colon \Sigma^2\to M^n$ of a compact oriented surface into a {\em
  conformal} manifold.  The conformal structure on $\Sigma$
induced by $x$ is equivalent to a complex structure $J_x$
given by anticlockwise rotation by $\pi/2$.  Consider the Lagrangian
\begin{equation*}
  \mathcal{L}\colon \mathcal{F}\to \mathcal{T},\quad
  \mathcal{L}(x)=[J_x].
\end{equation*}
It is shown in \cite{BohPetPin08} that the critical points
of $\mathcal{L}$, that is, the points $x$ where
$\d^{\mathcal{F}}_x\mathcal{L}$ fails to surject, are
precisely the globally isothermic immersion
$(x,q)\colon \Sigma\to M$. In more detail, from
\cite{BohPetPin08}, for $V\in\sect{x^*TM}$,
\begin{equation*}
  \d^{\mathcal{F}}_x\mathcal{L}(V)=[L_{V^{\top}}J_x+2A_0^{V^{\perp}}J_x].
\end{equation*}
Thus, $x\in\mathcal{F}$ is a
critical point of $\mathcal{L}$ if and only if there exists
a non-zero $\sigma\in T_{[J_x]}\mathcal{T}^{*}$ with
\begin{equation*}
  \int_{\Sigma} \ev{\sigma\wedge
  L_{V^{\top}}J_x+2A_0^{V^{\perp}}J_x}=0
\end{equation*}
for all $V\in\sect{x^*TN}$.  Hence the critical points of
our Lagrangian $\mathcal{L}$ are determined via the
variational operator
$\Dif_x\colon \sect{x^*TM}\to \Omega^2_{\Sigma}$ given by
\begin{equation*}
  \Dif_x(V)=\ev{\sigma\wedge
  L_{V^{\top}}J_x+2A_0^{V^{\perp}}J_x}.
\end{equation*}
Now $\ev{\sigma\wedge
  L_{V^{\top}}J_{x}}=2\ev{\sigma\wedge\nabla(J_xV^{\top})}$
so the Leibniz rule and $\d^{\nabla}\sigma=0$ give
\begin{equation*}
  \Dif_{x}(V)=-2\d\ev{\sigma,
  J_xV^{\top}}+2\ev{\sigma\wedge
  A^{V^{\perp}}_0J_x}.
\end{equation*}
Let us rewrite this in more familiar terms. Set
$Q=J_{x}^{t}\sigma$ and recall that $A_0$ anti-commutes with
$J_{x}$ to conclude that
\begin{equation*}
  \Dif_x(V)=-2\d\ev{Q,V^{\top}}-2\ev{Q\wedge A_{0}^{V^{\top}}}.
\end{equation*}
From this we read off the variational $1$-form
$\beta_x(V)=-2\ev{Q,V^{\top}}$ and the
Euler--Lagrange operator
$E_x(V)=-2\ev{Q\wedge
A^{V^{\perp}}_0}$. Thus we learn:
\begin{compactenum}
\item $x\in\mathcal{F}$ is a critical point of $\mathcal{L}$
  if and only if
  $\ev{Q\wedge A^{V^{\perp}}_0}=0$ for all
  $V\in\sect{x^*TM}$.  With $q=Q^{2,0}$, this is precisely
  the condition \eqref{eq:6} that $(x,q)$ be globally isothermic.
\item If $x$ is a critical point and $V=Y\circ x$ for a
  conformal vector field $Y\in\conf_{M}$ (which is
  clearly a
  symmetry for $\Dif_x$) then $\beta_x(Y\circ x)=-2\ev{Q,Y^{\top}}=-2\eta_Y$
  is essentially the retraction form and closed:
  \begin{equation*}
    \d\eta_{x}^q=0.
  \end{equation*}
\end{compactenum}
This gives another proof of \cref{th:4} and shows that, up to
a factor of $-\half$, $\eta^q_x$ is the Noether current of
our Lagrangian.



\begin{bibdiv}

  \begin{biblist}

\bib{Bia05}{article}{
  author={Bianchi, L.},
  title={Ricerche sulle superficie isoterme e sulla deformazione delle quadriche},
  date={1905},
  journal={Ann. di Mat.},
  volume={11},
  pages={93\ndash 157},
}

\bib{Bia05a}{article}{
  author={Bianchi, L.},
  title={Complementi alle ricerche sulle superficie isoterme},
  date={1905},
  journal={Ann. di Mat.},
  volume={12},
  pages={19\ndash 54},

}

\bib{BohPetPin08}{article}{
  author={Bohle, Christoph},
  author={Peters, G.~Paul},
  author={Pinkall, Ulrich},
  title={Constrained {W}illmore surfaces},
  date={2008},
  issn={0944-2669},
  journal={Calc. Var. Partial Differential Equations},
  volume={32},
  number={2},
  pages={263\ndash 277},
  doi={10.1007/s00526-007-0142-5},
  url={http://dx.doi.org/10.1007/s00526-007-0142-5},
  review={\MR {MR2389993 (2009a:53098)}},
}

\bib{bo60}{article}{
  author={Bonnet, Ossian},
  title={M\'emoire sur l’emploi d’un nouveau syst\`eme de variables dans l’\'etude des propri\'et\'es des surfaces courbes},
  journal={Journal de Math\'ematiques Pures et Appliqu\'ees},
  publisher={Gauthier-Villars},
  volume={2e s{\'e}rie, 5},
  date={1860},
  pages={156--266},
  url={http://portail.mathdoc.fr/JMPA/afficher_notice.php?id=JMPA_1860_2_5_A15_0},
}

\bib{Bou62a}{article}{
  author={Bour, Edmond},
  title={Th\'eorie de la d\'eformation des surfaces},
  date={1862},
  journal={J. L'\'Ecole Imp\'eriale Polytechnique},
  pages={1\ndash 148},
}

\bib{Bur06}{incollection}{
  author={Burstall, F.~E.},
  title={Isothermic surfaces: conformal geometry, {C}lifford algebras and integrable systems},
  date={2006},
  booktitle={Integrable systems, geometry, and topology},
  series={AMS/IP Stud. Adv. Math.},
  volume={36},
  publisher={Amer. Math. Soc.},
  address={Providence, RI},
  pages={1\ndash 82},
  review={\MR {MR2222512 (2008b:53006)}},
}

\bib{BurDonPedPin11}{article}{
  author={Burstall, Francis~E.},
  author={Donaldson, Neil~M.},
  author={Pedit, Franz},
  author={Pinkall, Ulrich},
  title={Isothermic submanifolds of symmetric {$R$}-spaces},
  date={2011},
  issn={0075-4102},
  journal={J. Reine Angew. Math.},
  volume={660},
  pages={191\ndash 243},
  doi={10.1515/crelle.2011.075},
  url={http://dx.doi.org/10.1515/crelle.2011.075},
  review={\MR {2855825}},
}

\bib{MR3523116}{article}{
  author={Burstall, F.},
  author={Hertrich-Jeromin, U.},
  author={M\"{u}ller, C.},
  author={Rossman, W.},
  title={Semi-discrete isothermic surfaces},
  journal={Geom. Dedicata},
  volume={183},
  date={2016},
  pages={43--58},
  issn={0046-5755},
  review={\MR {3523116}},
  doi={10.1007/s10711-016-0143-7},
}

\bib{BurHerPedPin97}{article}{
  author={Burstall, F.},
  author={Hertrich-Jeromin, U.},
  author={Pedit, F.},
  author={Pinkall, U.},
  title={Curved flats and isothermic surfaces},
  date={1997},
  issn={0025-5874},
  journal={Math. Z.},
  volume={225},
  number={2},
  pages={199\ndash 209},
  doi={10.1007/PL00004308},
  url={http://dx.doi.org/10.1007/PL00004308},
  review={\MR {MR1464926 (98j:53004)}},
}

\bib{BurSan12}{article}{
  author={Burstall, F.~E.},
  author={Santos, S.~D.},
  title={Special isothermic surfaces of type {$d$}},
  date={2012},
  issn={0024-6107},
  journal={J. Lond. Math. Soc. (2)},
  volume={85},
  number={2},
  pages={571\ndash 591},
  doi={10.1112/jlms/jdr050},
  url={http://dx.doi.org/10.1112/jlms/jdr050},
  review={\MR {2901079}},
}

\bib{Cal03}{article}{
  author={Calapso, P.},
  title={Sulle superficie a linee di curvatura isoterme},
  date={1903},
  journal={Rendiconti Circolo Matematico di Palermo},
  volume={17},
  pages={275\ndash 286},
}

\bib{CieGolSym95}{article}{
  author={Cie{\'s}li{\'n}ski, Jan},
  author={Goldstein, Piotr},
  author={Sym, Antoni},
  title={Isothermic surfaces in {$\mathbf {E}^3$} as soliton surfaces},
  date={1995},
  issn={0375-9601},
  journal={Phys. Lett. A},
  volume={205},
  number={1},
  pages={37\ndash 43},
  doi={10.1016/0375-9601(95)00504-V},
  url={http://dx.doi.org/10.1016/0375-9601(95)00504-V},
  review={\MR {MR1352426 (96g:53005)}},
}

\bib{Dar99e}{article}{
  author={Darboux, G.},
  title={Sur les surfaces isothermiques},
  date={1899},
  journal={C.R. Acad. Sci. Paris},
  volume={128},
  pages={1299\ndash 1305, 1538},
}

\bib{Dar99}{article}{
  author={Darboux, G.},
  title={Sur une classe de surfaces isothermiques li\'ees \`a la d\'eformations des surfaces du second degr\'e},
  date={1899},
  journal={C.R. Acad. Sci. Paris},
  volume={128},
  pages={1483\ndash 1487},
}

\bib{Dar99b}{article}{
  author={Darboux, Gaston},
  title={Sur les surfaces isothermiques},
  date={1899},
  issn={0012-9593},
  journal={Ann. Sci. \'Ecole Norm. Sup. (3)},
  volume={16},
  pages={491\ndash 508},
  url={http://www.numdam.org/item?id=ASENS_1899_3_16__491_0},
  review={\MR {MR1508975}},
}

\bib{MR1701599}{article}{
  author={Deligne, Pierre},
  author={Freed, Daniel S.},
  title={Classical field theory},
  conference={ title={Quantum fields and strings: a course for mathematicians, Vol. 1, 2}, address={Princeton, NJ}, date={1996/1997}, },
  book={ publisher={Amer. Math. Soc., Providence, RI}, },
  date={1999},
  pages={137--225},
  review={\MR {1701599}},
}

\bib{MR2827821}{article}{
  author={Fern\'{a}ndez, Isabel},
  author={Mira, Pablo},
  title={Constant mean curvature surfaces in 3-dimensional Thurston geometries},
  conference={ title={Proceedings of the International Congress of Mathematicians. Volume II}, },
  book={ publisher={Hindustan Book Agency, New Delhi}, },
  date={2010},
  pages={830--861},
  review={\MR {2827821}},
}

\bib{Her03}{book}{
  author={Hertrich-Jeromin, Udo},
  title={Introduction to {M}\"obius differential geometry},
  series={London Mathematical Society Lecture Note Series},
  publisher={Cambridge University Press},
  address={Cambridge},
  date={2003},
  volume={300},
  isbn={0-521-53569-7},
  review={\MR {MR2004958 (2004g:53001)}},
}

\bib{HerTjaZue97}{article}{
  title={On Guichard's nets and Cyclic systems},
  author={Hertrich-Jeromin, U.},
  author={Tjaden,E.-H.},
  author={Zuercher, M. T},
  year={1997},
  eprint={arXiv:dg-ga/9704003},
  doi={10.48550/arXiv.dg-ga/9704003},
}

\bib{MR1013786}{book}{
  author={Hopf, Heinz},
  title={Differential geometry in the large},
  series={Lecture Notes in Mathematics},
  volume={1000},
  edition={2},
  note={Notes taken by Peter Lax and John W. Gray; With a preface by S. S. Chern; With a preface by K. Voss},
  publisher={Springer-Verlag, Berlin},
  date={1989},
  pages={viii+184},
  isbn={3-540-51497-X},
  review={\MR {1013786}},
  doi={10.1007/3-540-39482-6},
}

\bib{Kap90}{article}{
  author={Kapouleas, Nicolaos},
  title={Complete constant mean curvature surfaces in {E}uclidean three-space},
  date={1990},
  issn={0003-486X},
  journal={Ann. of Math. (2)},
  volume={131},
  number={2},
  pages={239\ndash 330},
  doi={10.2307/1971494},
  url={http://dx.doi.org/10.2307/1971494},
  review={\MR {MR1043269 (93a:53007a)}},
}

\bib{KorKus93}{article}{
  author={Korevaar, Nick},
  author={Kusner, Rob},
  title={The global structure of constant mean curvature surfaces},
  date={1993},
  issn={0020-9910},
  journal={Invent. Math.},
  volume={114},
  number={2},
  pages={311\ndash 332},
  doi={10.1007/BF01232673},
  url={http://dx.doi.org/10.1007/BF01232673},
  review={\MR {MR1240641 (95f:53015)}},
}

\bib{KorKusSol89}{article}{
  author={Korevaar, Nicholas~J.},
  author={Kusner, Rob},
  author={Solomon, Bruce},
  title={The structure of complete embedded surfaces with constant mean curvature},
  date={1989},
  issn={0022-040X},
  journal={J. Differential Geom.},
  volume={30},
  number={2},
  pages={465\ndash 503},
  url={http://projecteuclid.org/getRecord?id=euclid.jdg/1214443598},
  review={\MR {MR1010168 (90g:53011)}},
}

\bib{MR1081331}{article}{
  author={Kusner, Rob},
  title={Bubbles, conservation laws, and balanced diagrams},
  conference={ title={Geometric analysis and computer graphics}, address={Berkeley, CA}, date={1988}, },
  book={ series={Math. Sci. Res. Inst. Publ.}, volume={17}, publisher={Springer, New York}, },
  date={1991},
  pages={103--108},
  review={\MR {1081331}},
  doi={10.1007/978-1-4613-9711-3_11},
}

\bib{MR3525098}{article}{
  author={Meeks, William H., III},
  author={P\'{e}rez, Joaqu\'{\i }n},
  author={Tinaglia, Giuseppe},
  title={Constant mean curvature surfaces},
  conference={ title={Surveys in differential geometry 2016. Advances in geometry and mathematical physics}, },
  book={ series={Surv. Differ. Geom.}, volume={21}, publisher={Int. Press, Somerville, MA}, },
  date={2016},
  pages={179--287},
  review={\MR {3525098}},
}

\bib{PalTer88}{book}{
  author={Palais, Richard~S.},
  author={Terng, Chuu-Lian},
  title={Critical point theory and submanifold geometry},
  series={Lecture Notes in Mathematics},
  publisher={Springer-Verlag},
  address={Berlin},
  date={1988},
  volume={1353},
  isbn={3-540-50399-4},
  review={\MR {MR972503 (90c:53143)}},
}

\bib{Sch01}{article}{
  author={Schief, W.~K.},
  title={Isothermic surfaces in spaces of arbitrary dimension: integrability, discretization, and {B}\"acklund transformations---a discrete {C}alapso equation},
  date={2001},
  issn={0022-2526},
  journal={Stud. Appl. Math.},
  volume={106},
  number={1},
  pages={85\ndash 137},
  doi={10.1111/1467-9590.00162},
  url={http://dx.doi.org/10.1111/1467-9590.00162},
  review={\MR {MR1805487 (2002k:37140)}},
}

\bib{Smy04}{article}{
  author={Smyth, Brian},
  title={Soliton surfaces in the mechanical equilibrium of closed membranes},
  date={2004},
  issn={0010-3616},
  journal={Comm. Math. Phys.},
  volume={250},
  number={1},
  pages={81\ndash 94},
  review={\MR {MR2092030 (2006f:53112)}},
  doi={10.1007/s00220-004-1085-8},
}


  \end{biblist}

\end{bibdiv}
\end{document}